\documentclass[a4paper,twoside]{article}
\usepackage[english]{babel}
\usepackage[latin1]{inputenc}
\usepackage{indentfirst}
\usepackage{amstext,amsfonts}
\usepackage{textcomp}
\usepackage{amssymb}
\usepackage{amscd}
\usepackage{amsthm,amsmath,amsxtra}
\usepackage{enumerate}

\DeclareMathOperator{\spn}{span}

\newcommand*{\<}{\langle}
\renewcommand{\>}{\rangle}
\newcommand*{\Real}{\mathbb{R}}

\newcommand*{\C}{\mathbb{C}}
\newcommand*{\Cl}{\mathcal{C}}
\newcommand*{\M}{\mathcal{M}}
\newcommand*{\Mat}{\mathbb{M}} 
\newcommand*{\N}{\mathcal{N}}
\newcommand*{\F}{\mathcal{F}}
\newcommand*{\G}{\mathcal{G}}
\newcommand*{\Ls}{\mathcal{L}}
\newcommand*{\K}{\mathcal{K}}
\newcommand*{\E}{\mathcal{E}}
\newcommand*{\la}{\lambda}
\newcommand*{\La}{\Lambda}
\newcommand*{\f}{\varphi}

\newcommand*{\sbe}{\subseteq}
\newcommand*{\dd}{\,\mathrm{d}}
\newcommand*{\dualg}{\widehat{G}} 
\newcommand*{\braket}[2]{\langle#1\!\mid\!#2\rangle}
\newcommand*{\id}{{\mathrm{id}}}
\newcommand*{\su}{{\mathrm{su}}}
\newcommand*{\st}{{\mathrm{s}}}
\newcommand*{\spettis}{{\mathrm{sp}}}
\newcommand*{\ii}{{\mathrm{i}}}
\newcommand*{\red}{\mathrm{r}} 
\newcommand*{\cont}{\mathcal{C}}
\newcommand*{\csg}{C_\red^*(G)}

\newtheorem{teor}{Theorem}[section]

\newtheorem{prop}[teor]{Proposition}
\newtheorem{defi}[teor]{Definition}
\newtheorem{obs}[teor]{Remark}
\newtheorem{coro}[teor]{Corollary}

\begin{document}

\thispagestyle{empty}

\begin{center}
{\Large \bf 
\vskip 0.5pc 
    A Generalized Fourier Inversion Theorem}
\vskip 3pc
{\large \bf Alcides Buss}\footnote{Supported by CAPES, Brazil.\newline
}
\vskip 1pc
\end{center}

\noindent
\textbf{Abstract.}
In this work we define operator-valued Fourier transforms for suitable integrable elements
with respect to the Plancherel weight of a (not necessarily Abelian) locally compact group.
Our main result is a generalized version of the Fourier inversion Theorem for
strictly-unconditionally integrable Fourier transforms.
Our results generalize and improve those previously obtained by Ruy Exel in the case of Abelian groups.


\section{Introduction}

Let $G$ be a locally compact Abelian group and let $\dualg$ be its Pontrjagin dual.
The classical Fourier inversion Theorem recovers, under certain conditions, a continuous integrable function
$f:G\to\C$ from its Fourier transform via the formula $f(t)=\int_{\dualg}\overline{\braket{\chi}{t}}\hat{f}(\chi)\dd{t}$,
where we write $\braket{\chi}{t}:=\chi(t)$ to emphasize the duality between $G$ and $\dualg$.
Here $\hat{f}(\chi):=\int_{G}\braket{\chi}{t}f(t)\dd{t}$ denotes the Fourier transform of $f$ and
we choose suitably normalized Haar measures $\dd{t}$ and $\dd{\chi}$ on $G$ and $\dualg$, respectively.

Ruy Exel \cite{exel_unconditional} extended the
classical Fourier inversion formula to operator-valued maps $f:G\to\Ls(H)$, where $H$ is a Hilbert space and $\Ls(H)$ denotes the
space of all bounded linear operators on $H$. He considered basically two generalized versions of Fourier's inversion
Theorem. The first one requires $f$ to be a
positive definite, weakly continuous, compactly supported function. The conclusion is that the Fourier transform $\hat{f}$ --
pointwise defined by the integral $\hat{f}(\chi):=\int_{G}\braket{\chi}{t}f(t)\dd{t}$ with respect to the strong operator topology -- is
unconditionally integrable with respect to the strong topology and its strong unconditional integral $\int_{\dualg}\overline{\braket{\chi}{t}}\hat{f}(\chi)\dd{\chi}$ equals
$f(t)$ for all $t\in G$. The second version requires $f$ to be a positive definite, strictly continuous, compactly supported function
$\dualg\to\M(A)$, where $A$ is now any $C^*$-algebra and $\M(A)$ is the multiplier algebra of $A$. Again, as
a conclusion one recovers $f(t)$ from the integral $\int_{\dualg}\overline{\braket{\chi}{t}}\hat{f}(\chi)\dd{\chi}$, but now all the integrals
are interpreted as strict unconditional integrals, that is, unconditional integrals with respect to the strict topology in $\M(A)$.

Both versions of Fourier's inversion Theorem considered above are equivalent. Indeed, one of the main tools used in \cite{exel_unconditional} is
Naimark's theorem on the structure of positive definite maps (see \cite[Theorem 3.2]{exel_unconditional}). It says that any
positive definite, weakly continuous map $f:G\to\Ls(H)$ has the form $f(t)=S^*u_t S$,
where $u$ is some strongly continuous unitary representation
of $G$ on a Hilbert space $H_u$ and $S:H\to H_u$ is some bounded linear operator. As a consequence any such map is automatically bounded and
strongly continuous. Moreover, it also implies that $f$ is strictly continuous if considered as a map $G\to \M\bigl(\K(H)\bigr)$, where
$\K(H)$ denotes the algebra of compact operators on $H$ and we identify $\M\bigl(\K(H)\bigr)\cong\Ls(H)$ in the canonical way.
Thus, if in addition $f$ is compactly supported, we can apply to $f$ the second version of Fourier's inversion Theorem for strictly
continuous maps mentioned above. Conversely, if $f:G\to\M(A)$ is a positive definite, strictly continuous,
compactly supported map, then we may view $f$ as a strongly continuous
map $G\to \Ls(H)$ and apply the first version, where $H$ is some Hilbert space endowed with a faithful nondegenerate representation of $A$.

What happens with the Fourier inversion Theorem if $G$ is not Abelian? The purpose of this paper is to answer this question.
We extend Exel's generalized version of Fourier's inversion Theorem to non-Abelian groups.
The starting point is to observe that the space of bounded, strictly continuous maps $\dualg\to\M(A)$ can be naturally identified with
the multiplier algebra $\M\bigl(A\otimes \csg\bigr)$, where $\csg$ denotes the reduced group $C^*$-algebra of $G$.
Here and throughout the rest of this paper, the symbol
$\otimes$ always denotes the \emph{minimal} tensor product.

Next, using the Plancherel weight on $\csg$ as a substitute for the classical Haar measure on $\dualg$ if $G$ is non-Abelian,
we define an appropriate subspace of \emph{integrable} elements in $\M\bigl(A\otimes \csg\bigr)$.
For each integrable element $a$, we define a (generalized) Fourier transform $\hat{a}$ which is a function on $G$ taking values in $\M(A)$.
As a conclusion, we prove that $a$ can be recovered from its Fourier transform via the strict unconditional integral
$a=\int_G \hat{a}(t)\otimes\la_t\dd{t}$, whenever this integral exists. The map $t\mapsto \la_t$ is the left regular representation of $G$ on
the Hilbert space $L^2(G)$ of square-integrable measurable functions on $G$: $\la_t(\xi)(s):=\xi(t^{-1}s)$ for all $\xi\in L^2(G)$ and $t,s\in G$.

Our version of the Fourier inversion Theorem can be interpreted as a generalization of Exel's version in \cite{exel_unconditional}.
Furthermore, our proof is considerably simpler than the original one in \cite{exel_unconditional}. While Exel's proof uses strong results like Naimark's theorem on the structure of positive definite maps and Stone's theorem on representations of locally compact Abelian groups, our proof basically only uses the definition.

\section{Weight theory}

One of the basic tools in this work is weight theory. In this section we recall some basic concepts, mainly to fix the notation. We refer to \cite{kustermans_vaes} for a detailed treatment.
Recall that a \emph{weight} on a $C^*$-algebra $C$ is a map $\f:C^+\to[0,\infty]$
that is additive and positively homogeneous, where $C^+$ denotes the set of positive elements in $C$.

We say that a positive element $x\in C^+$ is \emph{integrable} with respect to $\f$ if $\f(x)<\infty$.
We write $\M_{\f}^+$ for the set of positive integrable elements and $\N_\f$ for
the space $\{x\in C:x^*x\in\M_{\f}^+\}$ of \emph{square-integrable} elements.
Let $\M_{\f}$ be the linear span of $\M_{\f}^+$.
Then $\M_{\f}$ is a $*$-subalgebra of $C$,
$\N_\f$ is a left ideal of $C$ and
$\M_{\f}$ is the linear span of $\N_\f^*\N_\f=\{x^*y:x,y\in \N_\f\}$.

If $\M_{\f}^+$ is dense in $C^+$, then we say that $\f$ is
\emph{densely defined}. We also denote by $\f$ the unique linear
extension of $\f$ to $\M_{\f}$. We say that $\f$ is
\emph{lower semi-continuous} if $\{x\in C^+:\f(x)\leq c\}$ is closed for all $c\in\Real^+$ or,
equivalently,
for every net $(x_i)$ in  $C^+$ and $x\in C^+$, $x_i\to x$ implies $\f(x)\leq\liminf\bigl(\f(x_i)\bigr)$.

Define the sets
$\F_{\f}:=\{\omega\in C_+^*:\omega(x)\leq\f(x)\mbox{ for all } x\in C^+\}$
and
$\G_{\f}:=\{\alpha\omega:\omega  \in\F_{\f},\alpha\in (0,1)\}\sbe \F_{\f}.$
If we endow $\F_{\f}$ with the natural order of
$C^*_+$ then $\G_{\f}$ is a directed subset of
$\F_{\f}$, so that $\G_{\f}$ can be used as the index set
of a net. If $\f$ is lower semi-continuous, then (\cite[Theorem~1.6]{kustermans_vaes})
\begin{equation}\label{129}
\f(x)=\sup\{\omega(x):\omega\in\F_{\f}\}=\lim\limits_{\omega\in
\G_{\f}}\omega(x)\quad\mbox{for all }x\in \M_\f^+.
\end{equation}
Any lower semi-continuous weight $\f$ can be naturally extended to the multiplier
algebra $\M(C)$ by setting
$\bar{\f}(x):=\sup\{\omega(x):\omega\in\F_{\f}\}$ for all $x\in \M(\G)^+$,
where each $\omega\in C^*$ is extended to $\M(C)$ as usual. Then $\bar\f$ is the unique strictly
lower semi-continuous weight on $\M(C)$ extending $\f$.
We shall also denote the extension $\bar\f$ by $\f$ and use the notations
$\bar{\M}_{\f}^+=\M_{\bar{\f}}^+$,
$\bar{\M}_{\f}=\M_{\bar{\f}}$ and
$\bar{\N}_{\f}=\N_{\bar{\f}}$. Equation~\eqref{129} can be generalized:
$\f(x)=\lim\limits_{\omega\in\G_{\f}}\omega(x)$ for all $x\in\bar\M_{\f}$.

\subsection{Slicing with weights}

Let $A$ and $C$ be $C^*$-algebras.
Given a bounded linear functional $\theta$ on $A$,
we write $\theta\otimes\id$ for the canonical \emph{slice map} $A\otimes C\to C$. It is the unique bounded linear map satisfying the relation
$(\theta\otimes\id)(a\otimes x)=\theta(a)x$ for all $a\in A$ and $x\in C$. The map $\theta\otimes\id$ can be uniquely extended
to a strictly continuous map $\M\bigl(A\otimes C\bigr)\to\M\bigl(C\bigr)$, also denoted by $\theta\otimes\id$.

\begin{defi}\label{integrable_elements} Let $\f$ be a weight on $C$.
We say that a positive element $a\in \M\bigl(A\otimes C\bigr)^+$
is \emph{integrable \textup(with respect to the weight $\f$\textup)}, if there is $b\in\M(A)$ such that for every positive linear functional $\theta\in A^*_+$, $(\theta\otimes\id)(a)\in \bar\M_\f$ and $\f\bigl((\theta\otimes\id)(a)\bigr)=\theta(b)$.
\end{defi}

By Propositions~3.9 and~3.14 in \cite{kustermans_vaes}, $a\in \M\bigl(A\otimes C\bigr)^+$ is integrable if and only if
$a$ belongs to the set $\bar\M_{\id\otimes\f}^+$ of elements $a\in \M\bigl(A\otimes C\bigr)^+$ for which the net
\(\bigl((\id\otimes \omega)(x)\bigr)_{\omega\in\G_{\f}}\) converges strictly in  $\M(A)$.
Moreover, in this case the element $b\in \M(A)$ in Definition~\ref{integrable_elements} is given by $b=(\id\otimes\f)(a)$,
where we write $(\id\otimes\f)(a)$ for the strict limit of $\bigl((\id\otimes \omega)(a)\bigr)_{\omega\in\G_{\f}}$.
Let $\bar\M_{\id\otimes\f}$ be the linear span of $\bar\M_{\id\otimes\f}^+$ in $\M\bigl(A\otimes C\bigr)$. The map $\id\otimes\f$ has a
unique linear extension to $\bar\M_{\id\otimes\f}$, also denoted by $\id\otimes\f$. Elements in $\bar\M_{\id\otimes\f}$ are also called \emph{integrable}.

Let us assume that $C$ is commutative, that is, it has the form $C=\cont_0(X)$ for some locally compact topological space $X$, and suppose that $\f$ is the weight coming from a Radon measure $\mu$ on $X$. In other words, $\f$ is given by the integral $\f(f)=\int_X f(x)\dd{\mu(x)}$ for all $f\in \cont_0(X)^+$. In this case, the notion of integrability defined above recovers the usual notions of integrability for operator-valued functions on $X$. Indeed, first of all we may identify $\M(A\otimes C)$ with the $C^*$-algebra $\cont_b(X,\M^\st(A))$ of bounded
strictly continuous functions $f:X\to\M(A)$. Under this identification, we have the following result:

\begin{prop}\label{integrable_commutative} With the notations above, let $f$ be a positive element in $\M(A\otimes C)\cong\cont_b(X,\M^\st(A))$.
Then the following assertions are equivalent\textup:
\begin{enumerate}
\item[\textup{(i)}] $f$ is integrable in the sense of \textup{Definition~\ref{integrable_elements}};
\item[\textup{(ii)}]
             the net of strict Bochner integrals $\left(\int_X^\st f(x)\frac{\dd{\omega}}{\dd{\mu}}(x)\dd{\mu(x)}\right)_{\omega\in \G_\f}$
             converges strictly in $\M(A)$. Here $C^*_+=\cont_0(X)^*_+$ is identified with the space of positive bounded measures on $X$ and, for each $\omega\in \F_\f$, the symbol $\frac{\dd{\omega}}{\dd{\mu}}$ denotes the Radon-Nikodym derivative of $\omega$ with respect to $\mu$. Note that $\F_\f$ consists of the positive bounded measures $\omega$ that satisfy $\omega(E)\leq\mu(E)$ for every $\mu$-measurable subset $E\sbe X$. In particular, each $\omega\in \F_\f$ is absolutely continuous with respect to $\mu$ so that the Radon-Nikodym derivative $\frac{\dd{\omega}}{\dd{\mu}}$ is well-defined. Note also that $\frac{\dd{\omega}}{\dd{\mu}}$ is $\mu$-integrable and $0\leq\frac{\dd{\omega}}{\dd{\mu}}\leq 1$. Conversely, any such function gives rise to an element of $\F_\f$.
\item[\textup{(iii)}] the net of strict Bochner integrals $\left(\int_X^\st f(x)\omega_i(x)\dd{\mu(x)}\right)_{i\in I}$ converges strictly in $\M(A)$ for any net $\left(\omega_i\right)_{i\in I}$ of compactly supported continuous functions $\omega_i:X\to [0,1]$ for which $\omega_i(x)\to 1$ uniformly on compact subsets of $X$;
\item[\textup{(iv)}] the net of strict Bochner integrals $\left(\int_X^\st f(x)\omega_i(x)\dd{\mu(x)}\right)_{i\in I}$ converges strictly in $\M(A)$ for some net $\left(\omega_i\right)_{i\in I}$ as in \textup{(iii)};
\item[\textup{(v)}] $f:X\to \M(A)$ is strictly-unconditionally integrable, that is, the net of strict Bochner integrals
            $\left(\int_K^\st f(x)\dd{\mu(x)}\right)_{K\in \Cl}$ converges strictly in $\M(A)$, where $\Cl$ is the set of all $\mu$-measurable relatively compact subsets of $X$;
\item[\textup{(vi)}] $f:X\to \M(A)$ is strictly Pettis integrable, that is, for any $\mu$-measurable subset $E\sbe X$,
                there is an element $a_E\in \M(A)$ such
                that, for every continuous linear functional $\theta\in A^*$, the scalar valued function $\theta\circ f$ is $\mu$-integrable on $E$ in ordinary's sense, and $\int_E \theta(f(x))\dd{\mu(x)}=\theta(a_E)$;
\item[\textup{(vii)}] there is $a\in \M(A)$ such that for any positive linear functional $\theta\in A^*_+$, the scalar function $\theta\circ f$ is $\mu$-integrable on $X$ in ordinary's sense, and $\int_X \theta(f(x))\dd{\mu(x)}=\theta(a)$.
\end{enumerate}
In this event, we have
\begin{multline*}(\id\otimes\f)(f)=\st\mbox{-}\!\!\lim\limits_{\omega\in \G_\f}\int_X^\st f(x)\omega(x)\dd{\mu(x)}
        =\st\mbox{-}\!\lim\limits_{i\in I}\int_X^\st f(x)\omega_i(x)\dd{\mu(x)}\\
        =\int_X^\su f(x)\dd{\mu(x)}=\int_X^{\spettis} f(x)\dd{\mu(x)}=a.
\end{multline*}
The symbol $\int_X^\su$ above refers to \emph{strict unconditional} integrals and $\int_X^\spettis$ refers to \emph{strict Pettis} integrals.
\end{prop}
\begin{proof} As already noted above, (i) is equivalent to the fact that the net $\bigl((\id\otimes\omega)(f)\bigr)_{\omega\in \G_\f}$ converges strictly in $\M(A)$. Under the identification in (ii), each $(\id\otimes\omega)(f)$ corresponds to $\int_X^\st f(x)\frac{\dd{\omega}}{\dd{\mu}}(x)\dd{\mu(x)}$.
Thus (i) is equivalent to (ii). Item (vii) is just a reformulation of Definition~\ref{integrable_elements} because, under the identification $\M(A\otimes C)\cong\cont_b(X,\M^\st(A))$, the element $(\theta\otimes\id)(f)$ corresponds to composition $\theta\circ f$. Hence (i) is also equivalent to (vii). If $f$ is strictly-unconditionally integrable, then so is the pointwise product $\omega\cdot f$ for any bounded measurable scalar function $\omega:X\to\C$ (see \cite[Proposition 2.8]{exel_unconditional}). In particular, so is the restriction of $f$ to a $\mu$-measurable subset $E\sbe X$. From this, we see that (v) implies (vi). It is trivial that (vi) implies (vii). To see that (vii) implies (v), observe that because $f$ takes positive values, $\left(\int_K^\st f(x)\dd{\mu(x)}\right)_{K\in \Cl}$ is an increasing net of positive elements in $\M(A)$. 
By \cite[Lemma 3.12]{kustermans_vaes},
this net converges to some $a\in \M(A)$ if and only if
$\left(\int_K^\st \theta(f(x))\dd{\mu(x)}\right)_{K\in \Cl}$ converges to $\theta(a)$ for all $\theta\in A^*_+$. And this condition is equivalent to
(vii). We conclude that (i)$\Leftrightarrow$(ii)$\Leftrightarrow$(vii) and (v)$\Leftrightarrow$(vi)$\Leftrightarrow$(vii). The equivalences (iii)$\Leftrightarrow$(iv)$\Leftrightarrow$(v) follow from \cite[Proposition 12]{buss_meyer}. The last assertion is an easy consequence, whence the result.
\end{proof}

\begin{obs} It has been already observed by Ruy Exel in \cite{exel_unconditional,exel_spectral} that unconditional integrability is equivalent to Pettis integrability, at least for continuous operator-valued functions. A detailed proof of this fact in a more general context of functions defined on measure spaces and taking values in arbitrary Banach spaces can be found in the dissertation of Patricia Hess \cite[Teorema 4.14]{hess}.
The proof in \cite{hess} assumes $\sigma$-locality, which is a natural countability condition in measure-theoretical settings. Note that our proof above does not assume any countability condition. However, we are assuming strict continuity and positivity of our operator-valued function $f:X\to \M(A)$, and in particular our proof does not make sense in the general context of Banach spaces as in \cite{hess}.
\end{obs}

\subsection{The Plancherel weight}\label{S:plancherel_weight}

Let $G$ be a locally compact group. In this section, we collect some facts on
the Plancherel weight of the group von Neumann algebra $\Ls(G)$ of $G$. We refer to \cite[Section~7.2]{pedersen}
or \cite[Section~VII.3]{takesaki2} for a detailed construction.
Recall that the group von
Neumann algebra of $G$ is the von Neumann algebra $\Ls(G)=\csg''\sbe\Ls\bigl(L^2(G)\bigr)$ generated by
the left regular representation of $G$.

A function $\xi\in L^2(G)$ is called \emph{left bounded} if the map
$L^2(G)\supseteq\cont_c(G)\ni f\mapsto \xi*f\in L^2(G)$ extends to a bounded
operator on $L^2(G)$. In this case, we denote this operator by $\lambda(\xi)$.
Note that $\lambda(\xi)$ belongs to $\Ls(G)$ for every left bounded function
$\xi$. The Plancherel weight $\tilde\f:\Ls(G)^+\to [0,\infty]$ is defined by the formula
$$
\tilde\f(x):=\left\{
\begin{array}{cc}
\|\xi\|_2^2 & \mbox{if }x^{\frac{1}{2}}=\lambda(\xi) \mbox{ for some left bounded function }\xi\in L^2(G),\\
\infty & \mbox{otherwise}.\qquad\qquad\qquad\qquad\qquad\qquad\qquad\qquad\qquad\quad\,\,\,\,
\end{array}\right.
$$
We are mainly interested in the restriction of $\tilde\f$ to $\csg^+$, which we denote by $\f$.
It is a densely defined, lower semi-continuous weight on $\csg$.

From the definition of $\tilde\f$ above it follows that
$$\N_{\tilde\f}=\bigl\{\lambda(\xi):\xi\in L^2(G)\mbox{ is left bounded}\bigr\}$$
and (by polarization) $\tilde\f\bigl(\lambda(\xi)^*\lambda(\eta)\bigr)=\braket{\xi}{\eta}$ whenever $\xi,\eta\in L^2(G)$ are left bounded.
Here $\braket{\cdot}{\cdot}$ denotes the inner product on $L^2(G)$ (we assume it is linear on the second variable).
For functions $\xi$ and $\eta$ on $G$, we write $\xi*\eta$ and $\xi^*$ for the convolution
$\xi*\eta(t):=\int_G\xi(s)\eta(s^{-1}t)\dd{s}$ and the involution $\xi^*(t):=\Delta(t)^{-1}\overline{\xi(t^{-1})}$ whenever the operations
make sense. A short calculation shows that
($\xi^**\eta)(t)=\<\xi|V_t\eta\>$ for all $\xi,\eta\in L^2(G)$ and $t\in G$,
where $V_t(\eta)(s):=\eta(st)$. In particular, the function $\xi^**\eta$ is continuous and $(\xi^**\eta)(e)=\<\xi|\eta\>$, where $e$ denotes the identity element of $G$.
Thus, if $\xi,\eta\in L^2(G)$
are left bounded, the operator $\lambda(\xi^**\eta)=\lambda(\xi)^*\lambda(\eta)$ belongs to $\M_{\tilde{\f}}$ and
$\tilde{\f}\bigl(\lambda(\xi^**\eta)\bigr)=\braket{\xi}{\eta}=(\xi^**\eta)(e).$
We conclude that
$$\M_{\tilde\f}=\lambda\bigl(\cont_e(G)\bigr),$$
where $\cont_e(G):=\spn\{\xi^**\eta:\xi,\eta\in L^2(G)\mbox{ left
bounded}\}$, and $\tilde\f$ is given on functions of
$\cont_e(G)$ by evaluation at $e\in G$. Since $\f$ is the restriction of
$\tilde\f$ to $\csg$, we have $\bar\M_\f\sbe\M_{\tilde\f}$ and
the same formula holds for $\f$.

Finally, let us we remark that $\tilde\f$ is a \emph{KMS-weight} (see \cite{kustermans_vaes} for the definition of KMS-weights).
The \emph{modular automorphism group} $\{\sigma_x\}_{x\in \Real}$ of $\tilde\f$ is determined by
$\sigma_x(\lambda_t)=\Delta(t)^{\ii x}\lambda_t$ for all $t\in G$ and $x\in \Real$, where $\Delta$ is the modular function of $G$.
In particular, this implies that $\lambda_t$ is \emph{analytic} with respect to $\sigma$ -- meaning that the function $x\mapsto \sigma_x(\lambda_t)$
extends to an analytic function on $\C$. Its analytic extension is given by
\begin{equation}\label{E:analytic extension}
\sigma_z(\la_t)=\Delta(t)^{\ii z}\la_t\quad\mbox{for all }z\in \C\mbox{ and }t\in G.
\end{equation}

\begin{defi}\label{367} Given an integrable element $x\in \M_{\tilde\f}$, we define the \emph{Fourier transform} of $x$ to be
the function $\hat{x}:G\to\C$ given by $\hat{x}(t):=\tilde\f(\lambda_t^{-1}x)$ for all $t\in G$.
\end{defi}

Since $\lambda_t^{-1}=\lambda_{t^{-1}}$ is analytic with respect to the modular group of $\tilde\f$,
the element $\lambda_t^{-1}x$ belongs to $\M_{\tilde\f}$ whenever $x\in \M_{\tilde\f}$ (see \cite[Proposition 1.12]{kustermans_vaes}).
This fact can be also proved directly from the definition of $\tilde\f$ (see \cite[Proposition 2.8]{pedersen}).
Thus the Fourier transform $\hat{x}$ is well-defined.

If $G$ is Abelian, then under the isomorphism $\Ls(G)\cong L^\infty(\dualg)$, the Plancherel
weight on $\Ls(G)$ corresponds to the usual Haar integral on $L^\infty(\dualg)$. In this picture,
$\M_{\tilde\f}$ is identified with $L^\infty(\dualg)\cap L^1(\dualg)$ and
$\hat{x}$ corresponds to the Fourier transform
of the associated function in $L^\infty(\dualg)\cap L^1(\dualg)$.

\begin{prop}\label{369} Let $G$ be a locally compact group. Then the following properties hold\textup:
\begin{enumerate}
\item[\textup{(i)}] The Fourier transform $\hat{x}$ belongs to $\cont_e(G)$ for all $x\in \M_{\tilde\f}$.
In particular, $\hat{x}$ is a continuous function.
\item[\textup{(ii)}] The Fourier transform of $\lambda(f)$ is equal to $f$ for all $f\in \cont_e(G)$.
\item[\textup{(iii)}] If we equip $\cont_e(G)$ with the usual convolution of functions and the involution $f^*(t):=\Delta(t^{-1})\overline{f(t^{-1})}$, then $\cont_e(G)$ becomes a $*$-algebra
and the map
$$\M_{\tilde\f}\ni x\mapsto \hat{x}\in \cont_e(G)$$
is an isomorphism of $*$-algebras. The inverse is given by the map $f\mapsto \lambda(f)$. In particular, we have
$$(xy)\hat{}=\hat{x}*\hat{y},\quad\mbox{and}\quad (x^*)\hat{}=\hat{x}^*\quad\mbox{for all }x,y\in\M_{\tilde\f}.$$
\item[\textup{(iv)}] Suppose that $x\in \M_{\tilde\f}$ and that the function $t\mapsto \hat{x}(t)\lambda_t\in \Ls\bigl(L^2(G)\bigr)$
is integrable in the weak topology of $\Ls\bigl(L^2(G)\bigr)$. Then
$$\int_G^\mathrm{w}\hat{x}(t)\lambda_t \dd{t}=x,$$
where the superscript $``\mathrm{w}"$ above stands for integral in the weak topology.
\end{enumerate}
\end{prop}
\begin{proof}  We already know that $\M_{\tilde\f}=\lambda\bigl(\cont_e(G)\bigr)$. Let $x=\lambda(f)$ with $f\in \cont_e(G)$. Note that
$\lambda_t^{-1}x=\lambda_t^{-1}\lambda(f)=\lambda(f_t),$ where $f_t$ denotes the function $f_t(s):=f(ts)$. Hence
$\hat{x}(t)=\tilde\f\bigl(\lambda(f_t)\bigr)=f_t(e)=f(t)$, that is, $\hat{x}=f$.
This proves (i) and (ii). If $f,g,\xi,\eta\in L^2(G)$ are left bounded, then $(f^**g)*(\xi^**\eta)=(\lambda(g)^*f)^**(\lambda(\xi)^*\eta)$.
Note that, given $x\in \Ls(G)$ and $\zeta\in L^2(G)$ left bounded, $x\zeta\in L^2(G)$ is left bounded and $\lambda(x\zeta)=x\lambda(\zeta)$.
It follows that $(f^**g)*(\xi^**\eta)\in \cont_e(G)$. This shows that $\cont_e(G)$ is an algebra with convolution. Note also that $(f^**g)^*=g^**f\in \cont_e(G)$,
and therefore $\cont_e(G)$ is a $*$-algebra. It is easy to see that the map
$\M_{\tilde\f}\ni x\mapsto \hat{x}\in \cont_e(G)$ preserves the $*$-algebra structures.
For example, to prove that $(xy)\hat{}=\hat{x}*\hat{y}$, take $f,g\in \cont_e(G)$ such that $x=\lambda(f)$ and $y=\lambda(g)$.
Then $(xy)\hat{}=\bigl(\lambda(f*g)\bigr)\hat{}=f*g=\hat{x}\hat{y}$. Item (ii) and the fact that any
$x\in\M_{\tilde\f}$ has the form $x=\lambda(f)$ show that the map $x\mapsto\hat{x}$ has $f\mapsto\lambda(f)$ as its inverse.
Finally, we prove (iv). Take $\xi,\eta\in \cont_c(G)$. Then
\begin{align*}
\left\<\xi\Bigl|\left(\int_G^\mathrm{w}\hat{x}(t)\lambda_t\dd{t}\right)\eta\right\>&=\int_G\hat{x}(t)\braket{\xi}{\lambda_t(\eta)}\dd{t}
                                                                \\&=\int_G\int_G\hat{x}(t)\overline{\xi(s)}\eta(t^{-1}s)\dd{t}\dd{s}
																\\&=\int_G\overline{\xi(s)}(\hat{x}*\eta)(s)\dd{s}
																\\&=\braket{\xi}{\lambda(\hat{x})\eta}
																   =\braket{\xi}{x\eta}.\tag*{\qedhere}
\end{align*}
\end{proof}

\section{The Fourier transform}

Throughout the rest of this paper we fix a locally compact group $G$ and a $C^*$-algebra $A$.

\begin{defi}\label{009} Let $a\in \M\bigl(A\otimes\csg\bigr)$ be an integrable element. The
\emph{Fourier coefficient} of $a$ at $t\in G$ is the element $\hat{a}(t)\in \M(A)$ defined by
$$\hat{a}(t):=(\id\otimes\f)\bigl((1\otimes\la_t^{-1})a\bigr).$$
The map $t\mapsto \hat{a}(t)$ from $G$ to $\M(A)$ is called the \emph{Fourier transform} of $a$.
\end{defi}

As already observed, $\la_s$ is an analytic element for all $s\in G$.
This implies that $(1\otimes\la_s)x\in \bar\M_{\id\otimes\f}$ whenever $x\in \bar\M_{\id\otimes\f}$ (see \cite[Proposition 3.28]{kustermans_vaes}). Thus the Fourier transform is well-defined.

Suppose that the group $G$ is Abelian. Then there is a canonical isomorphism $\M\bigl(A\otimes \csg\bigr)\cong \cont_b\bigl(\dualg,\M^{\st}(A)\bigr)$,
the space of bounded, strictly continuous functions $\dualg\to\M(A)$.
Under this identification, $\bigl((1\otimes\lambda_{t^{-1}})a\bigr)(\chi)=\braket{\chi}{t}a(\chi)$ for all $a\in \cont_b\bigl(\dualg,\M^{\st}(A)\bigr)$
and $\chi\in \dualg$. Moreover, by Proposition~\ref{integrable_commutative}, a positive element $a\in \M\bigl(A\otimes\csg\bigr)$
is integrable if and only if there is $b\in \M(A)$ such that the function $t\mapsto \theta\bigl(a(\chi)\bigr)$ is integrable (in ordinary's sense)
and $\int_{\dualg}\theta(a(\chi))\dd{\chi}=\theta(b)$. It is also the content of Proposition~\ref{integrable_commutative} that
this notion of integrability is equivalent to Exel's notion of
strict unconditional integrability (essentially this fact has been also observed by Marc Rieffel; see \cite[Theorem 3.4, Proposition 4.4]{rieffel}).
In other words, $a\in\M\bigl(A\otimes\csg\bigr)^+$ is integrable if
and only if the corresponding function $\chi\mapsto a(\chi)$ in
$\cont_b\bigl(\dualg,\M^\st(A)\bigr)$ is strictly-unconditionally integrable. Furthermore,
in this case $(\id\otimes\f)(a)$ coincides with the strict unconditional integral $\int_{\dualg}^\su a(\chi)\dd{\chi}$.

We conclude that, if $G$ is Abelian, then the Fourier transform of an integrable element $a\in \M\bigl(A\otimes\csg\bigr)\cong\cont_b\bigl(\dualg,\M^\st(A)\bigr)$
coincides with the Fourier transform defined by Exel in \cite{exel_unconditional}:
$$\hat{a}(t)=\int_{\dualg}^\su\braket{\chi}{t} a(\chi)\dd{\chi}.$$

\section{Fourier inversion Theorem}

We are ready to prove the main result of this paper:

\begin{teor}[The Fourier inversion Theorem]\label{034} Let $G$ be a locally compact group and let $A$ be a $C^*$-algebra. Let $a\in \M\bigl(A\otimes \csg\bigr)$ be an integrable element and suppose that the function $G\ni t\mapsto \hat{a}(t)\otimes\la_t\in \M\bigl(A\otimes \csg\bigr)$ is strictly-unconditionally integrable. Then we have
$$a=\int_G^\su \hat{a}(t)\otimes\lambda_t\dd{t}.$$
\end{teor}
\begin{proof} Take any continuous linear functional $\theta\in A^*$ on $A$ and define the element $x:=(\theta\otimes\id)(a)\in \M\bigl(\csg\bigr)$.
Since $a$ is integrable, we have $x\in \bar\M_\f$. Moreover,
\begin{align*}
(\theta\otimes\id)\left(\int_G^\su\!\hat{a}(t)\otimes\lambda_t\dd{t}\right)&=\int_G^\su\!\theta\bigl(\hat{a}(t)\bigr)\lambda_t\dd{t}
                \\&=\int_G^\su\!\theta\Bigl((\id\otimes\f)\bigl((1_A\otimes\lambda_t^{-1})a\bigr)\Bigr)\lambda_t\dd{t}
															 \\&=\int_G^\su\!\f\bigl(\lambda_t^{-1}(\theta\otimes\id)(a)\bigr)\lambda_t\dd{t}
															 \\&=\int_G^\su\!\f(\lambda_t^{-1}x)\lambda_t\dd{t}
														     \\&=\int_G^\su\!\hat{x}(t)\lambda_t\dd{t}.
\end{align*}
Since strict convergence is stronger than weak convergence, the above equals $x=(\theta\otimes\id)(a)$ by Proposition~\ref{369}(iv).
The result follows because $\theta\in A^*$ is arbitrary.
\end{proof}

Theorem~\ref{034} extends Exel's operator-valued version of Fourier's inversion Theorem in \cite{exel_unconditional}
to non-Abelian groups. Assume that $G$ is Abelian.
Then, under the usual identification $\M\bigl(A\otimes \csg\bigr)\cong \cont_b\bigl(\dualg,\M^{\st}(A)\bigr)$, the element $\hat{a}(t)\otimes\lambda_t$ corresponds to the function
$\chi\mapsto \overline{\braket{\chi}{t}}\hat{a}(t)$. Thus Theorem~\ref{034} says that
$$\int_G^\su\overline{\braket{\chi}{t}}\hat{a}(t)\dd{t}=a(\chi)$$
whenever $a$ is integrable and the strict unconditional integral above exists.
The Fourier transform $\hat{a}$ in this case is
given by $\hat{a}(t)=\int_{\dualg}^\su\braket{\eta}{t}a(\eta)\dd{\eta}$.
Thus we can rewrite the equation above in the form of a generalized Fourier inversion formula:
$$\int_G^\su\overline{\braket{\chi}{t}}\left(\int_{\dualg}^\su\braket{\eta}{t}a(\eta)\dd{\eta}\right)\dd{t}=a(\chi).$$
As already mentioned in the introduction, Exel's version of Fourier's inversion Theorem starts with a
compactly supported, strictly continuous, positive definite function $f:G\to\M(A)$.
Apparently, our version requires no positivity condition on the functions involved. However, we are in fact
assuming a positivity condition because integrable elements are defined in terms of positive elements.

In order to compare our version with Exel's one, let us first recall that a function $f:G\to\M(A)$ is \emph{positive definite} if
for every finite subset $\{t_1,\ldots,t_n\}$ of $G$, the matrix $\bigl(f(t_i^{-1}t_j)\bigr)_{i,j}$ is positive in the
$C^*$-algebra $\Mat_n\bigl(\M(A)\bigr)$ of $n\times n$ matrices with entries in $\M(A)$. We may assume without loss of generality that
$A$ is a nondegenerate $C^*$-subalgebra of $\Ls(H)$ for some Hilbert space $H$.

The following result characterizes operator-valued, positive definite, weakly continuous functions.

\begin{prop}\label{FT:004} Let $A$ be a $C^*$-algebra which is faithfully and nondegenerately represented in $\Ls(H)$ for some Hilbert space $H$.
For a weakly continuous function $f:G\to \M(A)\sbe\Ls(H)$, the following assertions are equivalent\textup:
\begin{enumerate}
\item[\textup{(i)}] $f$ is positive definite;
\item[\textup{(ii)}] $f$ has the form $f(t)=S^*u_tS$ for some strongly continuous unitary representation $u:G\to\Ls(K)$
on some Hilbert space $K$ and some bounded linear operator $S:H\to K$;
\item[\textup{(iii)}] there is a strict completely positive map (see \cite{lance} for the precise definition) $F:C^*(G)\to \M(A)$ such that $f(t)=\tilde F(t)$, where
$\tilde F$ denotes the strictly continuous extension of $F$ to $\M(C^*(G))$ and we identify $G\sbe \M(C^*(G))$ in the usual way;
\item[\textup{(iv)}] $f$ has the form $f(t)=T^*w_t T$ for some strongly continuous unitary representation $w:G\to\Ls(\E)$ on some Hilbert
$A$-module $\E$ and some adjointable operator $T:A\to \E$.
\end{enumerate}
In this case, $f:G\to \M(A)$ is bounded and strictly continuous,
$f(e)$ is a positive operator, $f(t^{-1})^*=f(t)$ and $\|f(t)\|\leq\|f(e)\|$ for all $t\in G$.
Moreover, $f:G\to \Ls(H)$ is left strongly-uniformly continuous, that is, for $\xi\in H$,
$\|f(ts)\xi- f(t)\xi\|$ converges to zero uniformly in $t$ as $s$ converges to $e$.
\end{prop}
\begin{proof} The equivalence (i)$\Leftrightarrow$(ii) is Naimark's theorem (see \cite[Theorem 3.2]{exel_unconditional}).
Assume that (ii) holds. Then we can define $F(x):=S^*u(x)S$ for all $x\in C^*(G)$,
where we abuse the notation and write $u:C^*(G)\to \Ls(K)$ for the integrated form of $u:G\to\Ls(K)$.
Recall that $u(x)=\int_G x(t)u_t\dd{t}$ for all $x\in L^1(G)$. Note that $F(x)\in\M(A)$ because $f(t)\in \M(A)$ for all $t\in G$.
Since $u$ is a nondegenerate representation of $C^*(G)$, it follows that $F:C^*(G)\to \M(A)$
is a strict completely positive map \cite[Proposition 5.5]{lance}.  The strictly continuous extension of $F$ is given
by $\tilde F(x)=S^*\tilde u(x) S$, where $\tilde u:\M(C^*(G))\to\Ls(K)$ denotes the strictly continuous extension of $u$.
Hence $\tilde F(t)=S^*\tilde u(t) S=S^* u_t S=f(t)$ for all $t\in G$. Thus (ii) implies (iii). Now assume that (iii) is true.
Theorem 5.6 in \cite{lance} implies that there is a Hilbert $A$-module $\E$, a
nondegenerate $*$-homomorphism $w:C^*(G)\to \Ls(\E)$ and an adjointable operator $T:A\to \E$ such that $F(x)=T^*w(x) T$ for all $x\in C^*(G)$.
Defining $w_t:=\tilde w(t)$ to be the unitary representation of $G$ corresponding to $w$, we get item (iv). Finally, it is easy
to see that any function $f(t)=T^* w_t T$ as in (iv) is positive definite, so that (iv) implies (i). Therefore
all the four items are equivalent. The last assertion follows directly from (iv).
To prove the last assertion, take any $\xi\in H$. Using (ii), we get
$$\|f(ts)\xi-f(t)\xi\|\leq \|S\|\|u_s\eta-\eta\|$$
for all $t,s\in G$, where $\eta:=S\xi\in H$. Since $u$ is strongly continuous, it follows that $\|f(ts)\xi-f(t)\xi\|$ converges to zero
uniformly in $t$ as $s$ converges to $e$.
\end{proof}

\begin{obs}\label{warning on uniform continuity}
Let notation be as in Proposition~\ref{FT:004}. In general, it is not true that a weakly continuous, positive definite function
$f:G\to \Ls(H)$ is right strongly-uniformly continuous, that is, in general, given $\xi\in H$, $\|f(st)\xi-f(t)\xi\|$ does not converge to zero uniformly in $t$ as $s$ converges to $e$. Indeed, note that any unitary representation $u:G\to \Ls(H)$ is a positive definite function.
However, the left regular representation $\la:G\to \Ls(L^2(G))$ is not left strongly-uniformly continuous, unless $G$ has equivalent left and right uniform structures (see \cite[20.30]{hewitt_ross1}).
Of course, if the uniform structures of $G$ are equivalent, then the notions of left and right uniform continuity are equivalent, and we only speak of uniform continuity in this case meaning both left and right uniform continuity. Moreover, in this case, an analogous argument to that given in the proof of Proposition~\ref{FT:004} shows that any strictly continuous, positive definite function $f:G\to \M(A)$ is automatically (left and right) strictly-uniformly continuous, that is, for every $a\in A$,
all the expressions $\|f(ts)a-f(t)a\|$, $\|f(st)a-f(t)a\|$, $\|af(ts)-af(t)\|$ and $\|af(st)-af(t)\|$ converge to zero uniformly in $t$ as $s$ converges to $e$.
\end{obs}

Let $f\in \cont_c(G)$ and let $\rho(f)$ denote the operator on $L^2(G)$ given by right convolution with 
$f$: $\rho(f)\xi:=\xi*f$. Then $f$ is positive definite if and only the operator $\rho(f)$ is positive (see \cite[Proposition 7.1.9]{pedersen}).
Moreover, it is easy to see that $\rho(f)=J\la(Jf)J$, where $J$ is the anti-unitary operator on $L^2(G)$ defined by $J\xi(t):=\Delta(t)^{-\frac{1}{2}}\overline{\xi(t^{-1})}$. It follows that $f$ is positive definite if and only if $\la(Jf)$ is a positive operator.
Note that $Jf=\Delta^{-\frac{1}{2}}\cdot f$ if $f$ is positive definite. In particular, if $G$ is unimodular, $f$ is positive definite
if and only if $\la(f)$ is positive. In general, $\la(f)$ is positive if and only if $\Delta^{\frac{1}{2}}\cdot f$ is positive definite.

Lemma 7.2.4 in \cite{pedersen} shows that a function $f\in \cont_c(G)$ is positive definite if and only if $f=\eta*\tilde\eta$,
where $\eta$ is some right bounded function in $L^2(G)$ and $\tilde\eta(t):=\overline{\eta(t^{-1})}$ for all $t\in G$. Recall that
a function $\eta\in L^2(G)$ is called \emph{right bounded} if the map
$L^2(G)\supseteq\cont_c(G)\ni g\mapsto g*\eta\in L^2(G)$ extends
to a bounded operator on $L^2(G)$. Alternatively, $\eta$ is right bounded if and only if $J\eta$ is left bounded. This follows from the relation
$J(g*\eta)=(J\eta)*(Jg)$. Using the easily verified relation $J(\eta*\tilde\eta)=(J\eta)^**(J\eta)$ and the fact that
$\eta$ is right bounded if and only if $J\eta$ is left bounded, we get that $\la(f)$ is a positive operator
if and only if $f=\xi^**\xi$ for some left bounded function $\xi\in L^2(G)$. In particular, $f\in \cont_e(G)$ so that $\la(f)\in \M_\f$.
This proves the following result:

\begin{prop}\label{positive definite are integrable} Let $f$ be a function in $\cont_c(G)$.
If $\Delta^{\frac{1}{2}}\cdot f$ is a positive definite function, that is, if $\la(f)$ is a positive operator on $L^2(G)$,
then $\la(f)$ is integrable with respect to the Plancherel weight $\f$, that is, $\la(f)\in \M_\f$.
\end{prop}

\begin{obs} Given $f\in \cont_c(G)$, it is not true in general that $\la(f)\in \M_\f$.
Indeed, assume that $G$ is compact so that $\cont_c(G)=\cont(G)$.
Then the inclusion $\la\bigl(\cont(G)\bigr)\sbe\M_\f\sbe\la\bigl(\cont_e(G)\bigr)$ implies
$\cont_e(G)=\cont(G)$ because we always have $\cont_e(G)\sbe\cont(G)$.
Since $G$ is compact, $L^2(G)\sbe L^1(G)$ and therefore any function in $L^2(G)$ is left bounded. Thus $\cont_e(G)$
equals the linear span of $L^2(G)*L^2(G)$. Therefore, the inclusion $\la\bigl(\cont(G)\bigr)\sbe\M_\f$ implies that $\cont(G)$ equals
the linear span of $L^2(G)*L^2(G)$. This is true if only if $G$ is finite \textup{\cite[34.40]{hewitt_ross1}}.
Hence, if $G$ is a compact infinite group, $\la\bigl(\cont(G)\bigr)$ is not contained in $\M_\f$.
\end{obs}

Proposition~\ref{positive definite are integrable} can be generalized to operator-valued functions.
First, we have to extend the left regular representation to operator-valued functions: given a $C^*$-algebra $A$,
there is a canonical map $\la_A$ from $\cont_c\bigl(G,\M^\st(A)\bigr)$ into $\M\bigl(A\otimes \csg\bigr)$ that coincides with the left regular representation $\la:C_c(G)\to\Ls(L^2(G))$ if $A=\C$. In fact, assume that $A$ is a nondegenerate
$C^*$-subalgebra of $\Ls(H)$, so that $A\otimes \csg$ -- and so also $\M\bigl(A\otimes\csg\bigr)$ -- is a nondegenerate $C^*$-subalgebra of $\Ls\bigl(H\otimes L^2(G)\bigr)\cong \Ls\bigl(L^2(G,H)\bigr)$.
The map $\la_A$ is then given by
$$\bigl(\la_A(f)\xi\bigr)(t)=(f*\xi)(t):=\int_G f(s)\xi(s^{-1}t)\dd{s}$$
for all $f\in \cont_c\bigl(G,\M^\st(A)\bigr)$, $\xi\in \cont_c(G,H)$ and $t\in G$.

\begin{prop} Let $f$ be a function in $\cont_c\bigl(G,\M^\st(A)\bigr)$. Then $\la_A(f)$ is a positive operator if
and only if the pointwise product $\Delta^{\frac{1}{2}}\cdot f$ is a positive definite function.
Moreover, in this case $a:=\la_A(f) \in\M\bigl(A\otimes\csg\bigr)$ is an integrable element and $\hat{a}=f$.
In particular, we have the following formula for $\la_A(f)$\textup:
$$\la_A(f)=\int_G^\su f(t)\otimes\la_t\dd{t}.$$
\end{prop}
\begin{proof} If $\theta\in A^*_+$ is a positive linear functional on $A$, then a straightforward calculation shows that
$(\theta\otimes\id)\bigl(\la_A(f)\bigr)=\la(\theta\circ f)$, where $(\theta\circ f)(t):=\theta\bigl(f(t)\bigr)$. Hence,
$\la_A(f)\geq 0$ if and only if $(\theta\otimes\id)\bigl(\la_A(f)\bigr)=\la(\theta\circ f)\geq 0$ for all  $\theta\in A^*_+$.
By the discussion preceding Proposition~\ref{positive definite are integrable}, $\la_A(f)$ is positive if and only if
$\Delta^{\frac{1}{2}}\cdot(\theta\circ f)=\theta\circ(\Delta^{\frac{1}{2}}\cdot f)$ is positive definite for all $\theta\in A^*_+$.
And this is equivalent to $\Delta^{\frac{1}{2}}\cdot f$ being positive definite. This proves the first assertion.
By Proposition~\ref{positive definite are integrable}, $(\theta\otimes\id)(a)=\la(\theta\circ f)\in \M_\f$ for all $\theta\in A^*_+$,
and $$\f\bigl((\theta\otimes\id)(a)\bigr)=\f\bigl(\la(\theta\circ f)\bigr)=(\theta\circ f)(e)=\theta\bigl(f(e)\bigr).$$ This shows that $a$ is integrable and $(\id\otimes\f)(a)=f(e)$.
Moreover, Proposition~\ref{369}(ii) yields
\begin{align*}
\theta\bigl(\hat{a}(t)\bigr)&=\theta\bigl((\id\otimes\f)((1\otimes\la_t^{-1})a)\bigr)=\f\bigl((\theta\otimes\id)((1\otimes\la_t^{-1})a)\bigr)
    \\&=\f\bigr(\la_t^{-1}(\theta\otimes\id)(a)\bigr)=\f\bigl(\la_t^{-1}\la(\theta\circ f)\bigr)=(\theta\circ f)(t)=\theta\bigl(f(t)\bigr).
\end{align*}
Since $\theta$ is arbitrary, we get $\hat{a}=f$. The final assertion follows from Theorem~\ref{034} because
$f$ is strictly Bochner integrable and hence also strictly-unconditionally integrable.
\end{proof}

\section{Further properties of the Fourier transform}

In this section we analyze some additional properties of the Fourier transform $t\mapsto\hat{a}(t)$ of an integrable
element $a\in \M\bigl(A\otimes\csg\bigr)$. We prove that $\hat{a}$ is always a strictly continuous
function and that $\Delta^{\frac{1}{2}}\cdot\hat{a}$ is a positive definite function if $a$ is a positive integrable element.

First, we need some preparation. We say that $a\in \M\bigl(A\otimes\csg\bigr)$ is \emph{square-integrable}
(with respect to the Plancherel weight $\f$) if $a^*a$ is an integrable element. Let $\bar\N_{\id\otimes\f}$ be the
space of square-integrable elements in $\M\bigl(A\otimes\csg\bigr)$. Then $\bar\N_{\id\otimes\f}$ is a right
ideal in $\M\bigl(A\otimes\csg\bigr)$ and the space of integrable elements $\bar\M_{\id\otimes\f}$ is
the linear span of $\bar\N_{\id\otimes\f}^*\bar\N_{\id\otimes\f}=\{a^*b:a,b\in \bar\N_{\id\otimes\f}\}$.

Recall that a \emph{GNS-construction} for a weight $\f$ on a $C^*$-algebra $C$
is a triple $(K,\pi,\La)$, where $K$ is some Hilbert space,
$\La:\N_\f\to K$ is a linear map with dense image satisfying $\f(a^*b)=\braket{\La(a)}{\La(b)}$ for all $a,b\in \N_\f$,
and $\pi:C\to\Ls(K)$ is a $*$-representation of $C$ satisfying $\pi(a)\La(b)=\La(ab)$ for all $a\in C$ and $b\in \N_\f$.
A GNS-construction always exists and is unique up to unitary transformation.

There is a canonical GNS-construction for the Plancherel weight $\f$ on $\csg$ given by $(L^2(G),\iota,\La)$,
where $\iota$ denotes the inclusion map $\csg\hookrightarrow\Ls\bigl(L^2(G)\bigr)$ and
$\La(\la(\xi))=\xi$ for every left bounded function $\xi\in L^2(G)$ with $\la(\xi)\in \csg$.
We always use the GNS-construction $(L^2(G),\iota,\La)$ for $\f$.

The GNS-map $\La:\N_\f\to L^2(G)$ can be naturally extended to a linear map $\id\otimes\La:\bar\N_{\id\otimes\f}\to\Ls\bigl(A,L^2(G,A)\bigr)$.
Here $L^2(G,A)\cong A\otimes L^2(G)$ denotes the Hilbert $A$-module defined as the completion of $\cont_c(G,A)$ with
respect to the inner product $\braket{f}{g}_A:=\int_Gf(t)^*g(t)\dd{t}$ and the canonical right $A$-action.
The space $\Ls\bigl(A,L^2(G,A)\bigr)$ is set of all adjointable maps $A\to L^2(G,A)$, where we view $A$ as a Hilbert $A$-module in the obvious way.
The map $\id\otimes\La$ is characterized by the equation
$(\id\otimes\La)(a)^*\bigl(b\otimes\La(x)\bigr)=(\id\otimes\f)\bigl(a^*(b\otimes x)\bigr)$ for all $a\in \bar\N_{\id\otimes\f}$, $b\in A$ and $x\in \N_\f$. We refer to \cite{kustermans_vaes} for more details on the construction and properties
of the map $\id\otimes\La$. One of its basic properties is the relation
\begin{equation}\label{001}
(\id\otimes\La)(a)^*(\id\otimes\La)(b)=(\id\otimes\f)(a^*b)\quad\mbox{for all }a,b\in \bar\N_{\id\otimes\f}.
\end{equation}

\begin{prop}\label{002} Let $a\in \M\bigl(A\otimes\csg\bigr)$ be an integrable element and let $b_i,c_i\in \M\bigl(A\otimes\csg\bigr)$ be
square-integrable elements with $a=\sum\limits_{i=1}^n b_i^*c_i$. Then
$$\hat{a}(t)=\sum\limits_{i=1}^n(\id\otimes\La)(b_i)^*V_t(\id\otimes\La)(c_i)\quad\mbox{for all }t\in G,$$
where $V:G\to \Ls\bigl(L^2(G,A)\bigr)$ is the representation of $G$ defined by $V_t(f)(s):=f(st)$ for all $f\in \cont_c(G,A)$ and $t,s\in G$.
\end{prop}
\begin{proof} It is enough to consider $a$ of the form $a=b^*c$ with $b,c$ square-integrable.
Equation~\eqref{001} yields
$$\hat{a}(t)=(\id\otimes\f)\bigl((b(1\otimes\la_t))^*c\bigr)=(\id\otimes\La)\bigl(b(1\otimes\la_t)\bigr)^*(\id\otimes\La)(c).$$
Since $\la_t$ is analytic with respect to the modular automorphism group $\sigma$ of $\f$ (see Section~\ref{S:plancherel_weight}),
it follows from \cite[Proposition 3.28]{kustermans_vaes} that $b(1\otimes\la_t)$ is square-integrable and
$(\id\otimes\La)\bigl(b(1\otimes\la_t)\bigr)=(1\otimes J\sigma_{\frac{\ii}{2}}(\la_t)^*J)(\id\otimes\La)(b)$,
where $J$ is the \emph{modular conjugation} of $\f$ in the GNS-construction $(L^2(G),\iota,\La)$.
It remains to show that $1\otimes J\sigma_{\frac{\ii}{2}}(\la_t)J=V_t$ for all $t\in G$.
Equation~\eqref{E:analytic extension} implies $\sigma_{\frac{\ii}{2}}(\la_t)=\Delta(t)^{-\frac{1}{2}}\la_t$.
The modular conjugation is
given by $(J\xi)(s)=\Delta(s)^{-\frac{1}{2}}\overline{\xi(s^{-1})}$ for all $\xi\in L^2(G)$ and $s\in G$.
The desired relation $1\otimes J\sigma_{\frac{\ii}{2}}(\la_t)J=V_t$ now follows.
\end{proof}

\begin{coro}\label{FT:003} Let $a\in \M\bigl(A\otimes\csg\bigr)$ be an integrable element.
Then the Fourier transform $\hat{a}$ is a strictly continuous function $G\to \M(A)$.

If $a$ is positive, then the pointwise product $\Delta^{\frac{1}{2}}\cdot\hat{a}$ is a positive definite function. In general,
$\Delta^{\frac{1}{2}}\cdot\hat{a}$ is a linear combination of positive definite
functions.
\end{coro}
\begin{proof}
Note that $\rho_t:=\Delta(t)^{\frac{1}{2}}V_t$ is the right regular representation
of $G$ on $L^2(G,A)$.
Any integrable element is, by definition, a linear combination of positive integrable elements.
If $a$ is a positive integrable element,
then $a=b^*b$ for some square-integrable element $b\in \M\bigl(A\otimes\csg\bigr)$; take for instance $b=a^{\frac{1}{2}}$.
Proposition~\ref{002} implies that $\Delta(t)^{\frac{1}{2}}\cdot\hat{a}(t)=S^*\rho_t S$, where $S:=(\id\otimes\La)(b)\in \Ls\bigl(A,L^2(G,A)\bigr)$.
Since $\rho$ is a strongly continuous unitary representation of $G$,
functions of the form $t\mapsto S^*\rho_t S$ are positive definite and strictly continuous.
\end{proof}

\begin{coro} Assume that $A$ is faithfully and nondegerately represented in $\Ls(H)$ for some Hilbert space $H$.
If $a\in \M(A\otimes\csg)$ is an integrable element, then $\Delta^{\frac{1}{2}}\cdot\hat{a}:G\to \M(A)\sbe\Ls(H)$
is a left strongly-uniformly continuous function. Moreover, if $G$ has equivalent uniform structures, then $\hat{a}$
is a strictly-uniformly continuous function $G\to\M(A)$.
\end{coro}
\begin{proof} By Corollary~\ref{FT:003}, $\Delta^{\frac{1}{2}}\cdot\hat{a}$ is a linear combination of strictly continuous
positive definite functions $G\to\M(A)$. Proposition~\ref{FT:004} yields the first assertion.
The final assertion follows from Remark~\ref{warning on uniform continuity}.
Note that $G$ is unimodular if it has equivalent uniform structures \cite[19.28]{hewitt_ross1}.
\end{proof}

\noindent
\textbf{Alcides Buss}\newline
Mathematisches Institut\newline
Westf\"alisches Wilhelms-Universit\"at M\"unster\newline
Einsteinstraße\ 62,\, 48149 M\"unster,\, Germany
\vskip 1pc
\noindent
E-mail: abuss@math.uni-muenster.de

\end{document}